\documentclass[12pt]{article}
\usepackage[utf8]{inputenc}
\usepackage{amsmath}
\usepackage{physics}
\usepackage{graphicx}
\usepackage{mathtools}
\usepackage{amssymb}
\usepackage{bm}

\usepackage{amsthm}
\usepackage{float}
\usepackage{hyperref}

\usepackage{tikz}
\newcommand{\NN}{\mathbb{N}}

\newcommand{\ZZ}{\mathbb{Z}}

\newcommand{\Pc}{\mathcal{P}}

\newcommand{\Fc}{\mathcal{F}}

\theoremstyle{definition}
\newtheorem{theorem}{Theorem}

\newtheorem{definition}[theorem]{Definition}
\newtheorem{lemma}[theorem]{Lemma}
\newtheorem{remark}[theorem]{Remark}

\newtheorem{example}[theorem]{Example}

\tikzset{every picture/.style={line width=0.75pt}} 

\title{Extremal Uniquely Resolvable Multisets}
\author{Varun Sivashankar\\ 
\href{mailto:varunsiva@ucla.edu}{\texttt{varunsiva@ucla.edu}}\\
\normalsize UCLA Mathematics}
\date{}
\begin{document}

\maketitle

\begin{abstract}
For positive integers $n$ and $m$, consider a multiset of non-empty subsets of $[m]$ such that there is a \textit{unique} partition of these subsets into $n$ partitions of $[m]$. We study the maximum possible size $g(n,m)$ of such a multiset. We focus on the regime $n \leq 2^{m-1}-1$ and show that $g(n,m) \geq \Omega(\frac{nm}{\log_2 n})$. When $n = 2^{cm}$ for any $c \in (0,1)$, this lower bound simplifies to $\Omega(\frac{n}{c})$, and we show a matching upper bound $g(n,m) \leq O(\frac{n}{c}\log_2(\frac{1}{c}))$ that is optimal up to a factor of $\log_2(\frac{1}{c})$. We also compute $g(n,m)$ exactly when $n \geq 2^{m-1} - O(2^{\frac{m}{2}})$.

\end{abstract}

\section{Introduction}

In this paper, we study the maximum size of a multiset that admits a certain kind of unique partition. This problem was originally motivated by attempting to count the minimum number of rules required to guarantee the existence of a unique solution for a logic based problem called the zebra puzzle.

Fix $n,m \in \ZZ^+$. Let $\Fc = \{C_1,\ldots,C_k\}$ be a multiset consisting of non-empty subsets of $[m] = \{1,\ldots,m\}$. Suppose there exists a \textit{unique} partition $\mathcal{P} = \{A_1,\ldots,A_n\}$ of $[k]$ such that for each $i \in [n]$, $\sqcup_{j \in A_i} C_j = [m]$ is a disjoint partition of $[m]$.

\[\textbf{Question: } \textit{What is the size $g(n,m)$ of the largest such multiset?}\]

This question is related to the study of block designs. We refer the reader to \cite{block} for a thorough treatment. Here, we include some basic definitions. Let $X$ be a finite set. A block design $(X,\Fc)$ on $X$ is a collection $\mathcal{F}$ of subsets of $X$. $(X,\Fc)$ is said to be $t$-\textit{balanced} if every $t$-subset of $X$ is a subset of exactly $\lambda$ blocks in $\mathcal{F}$. Block designs are almost always assumed to be \textit{uniform}, which means that every set in $\Fc$ must have the same cardinality. However, \textit{non-uniform} block designs are also studied. For example, pairwise balanced designs studied in \cite{pbd} are block designs with $t=2$ and non-uniform block sizes. Lastly, a block design is called \textit{resolvable} if we can partition $\Fc$ into classes such that the sets in each class form a partition of $X$ \cite{block}.

Let $X = [m]$ and $\Fc$ a collection of non-empty subsets of $[m]$. Our question is related to 1-balanced block designs, where every element $i \in [m]$ appears in $\lambda = n$ subsets of $[m]$. Further, we require $(X,\Fc)$ to be \textit{uniquely} resolvable. As we will see later, in our regime of $n \leq 2^{m-1}-1$, we need not include $[m]$ in $\Fc$, and no other set can be repeated without violating unique resolvability. To summarize, for parameters $n,m \in \NN$, our problem can be viewed as determining the largest possible size of a non-uniform 1-balanced block design on $[m]$ with $\lambda = n$ that is uniquely resolvable. 

However, since prior work has focused on uniform designs, we will merely borrow some terms such as resolvability for convenience, as opposed to viewing our problem as a question in block design. 

The zebra puzzle is often used as a benchmark in testing algorithms for constraint satisfaction problems. The problem states that there are 5 people, each with a distinct attributes like nationality, color of house, preferred drink, preferred cigarette brand and pet. Examples of constraints include `The Englishman lives in the Red house', `Coffee is drunk in the Green house' and `The Spaniard owns the dog'. The puzzle then asks for the identity of the person with a zebra as a pet \cite{wiki}. Given enough rules, it is possible to determine exactly the attributes of every person. Motivated by this problem, we consider the minimum number of rules required to uniquely specify a solution for an arbitrary instance of the problem with $n$ people and $m$ attributes. In our formulation, we only permit rules that can be expressed as $2$-tuples, such as (Coffee, Green House) or (Spaniard, Dog) and do not permit compound rules. The minimum number of rules required in this formulation is $nm - g(n,m)$, where $g(n,m)$ is the size of the multiset discussed above. We provide a small example below to illustrate the correspondence.

\begin{figure}[H]
    \centering
    \includegraphics[scale=0.7]{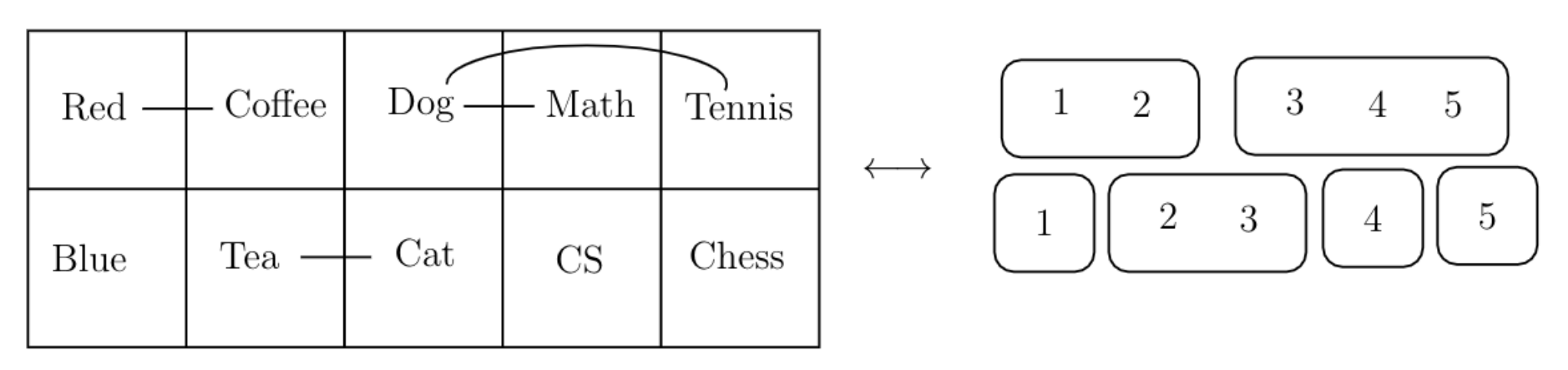}
    \caption{This corresponds to the case $(n,m) = (2,5)$. The rules of the puzzle here are (Red,Coffee), (Dog,Math), (Tennis,Dog) and (Tea,Cat). We number the attributes Color, Drink, Pet, Subject, Sport from 1 to 5 and create the corresponding multiset $\Fc = \{\{1,2\},\{3,4,5\},\{1\},\{2,3\},\{4\},\{5\}\}$. Given these rules, we know that one person must have the attributes (Red,Coffee,Dog,Math,Tennis) and the other person must have (Blue,Tea,Cat,CS,Chess). If we have an optimal number of rules that guarantees a unique solution to the puzzle, in any row with $k$ rules, we will have $m-k$ sets in the corresponding multiset diagram. So if $g(n,m)$ maximizes the number of sets that can be uniquely partitioned, it is easy to see why the optimal number of rules must be $nm-g(n,m)$.}
    \label{fig1}
\end{figure}

While this puzzle provides some motivation, the question itself is a very natural extremal combinatorics problem. The rest of this paper will not assume any knowledge about the puzzle. 

\subsection{Main Results}

We first prove a simple result that $g(2,m) = m+1$ (Theorem 12) and introduce the Subset Criterion (Lemma 13), a useful tool for further analysis. The subset criterion reveals that for $n > 2^{m-1} - 1$, any extremal multiset must have at least $n-2^{m-1}+1$ copies of the whole set $[m]$ to maintain uniqueness, which is not particularly exciting. We focus on the more interesting regime of $n \leq 2^{m-1} - 1$, for which we establish the following results.

\begin{enumerate}
\item[(i)] \textbf{Theorem 23:} Suppose $n \leq 2^{m-1}-1$. Then
$g(n,m) > \frac{n(m+1)}{\log_2(n+1)+2} = \Omega\left(\frac{nm}{\log_2 n}\right)$.
When $n = 2^{cm}$ for any $c \in (0,1)$, this simplifies to 
$g(n,m) \geq \Omega\left(\frac{n}{c}\right)$.
\item[(ii)] \textbf{Theorem 27:} Suppose $n = 2^{cm} \leq 2^{m-1}-1$ for any $c \in (0,1)$. Then
$g(n,m) \leq \frac{n}{c} \left(6 - 3.2\log_2(c)\right) = O\left(\frac{n}{c}\log_2\left(\frac{1}{c}\right)\right)$. So the upper bound matches the lower bound up to a $\log_2(\frac{1}{c})$ factor.
\item[(iii)] \textbf{Theorem 29:} Suppose $2^{m-1}+1-2^{\lfloor\frac{m}{2}\rfloor}\leq n \leq 2^{m-1} - 1$. Then $g(n,m) = 2n + \lfloor \frac{2^{m-1}-1-n}{2} \rfloor$.
\end{enumerate}

\section{Preliminaries}
\begin{definition}[Resolvable Multiset]
Fix $n,m \in \ZZ^+$. $\mathcal{F}$ is a resolvable multiset if it is a multiset of non-empty proper subsets of $[m]$, say $\mathcal{F} = \{C_1,\ldots,C_k\}$, such that there exists a partition $\mathcal{P} = \{A_1,\ldots,A_n\}$ of $[k]$ with the property that $\sqcup_{j \in A_i} C_j = [m]$ is a partition of $[m]$ for all $i=1,\ldots,n$.
\end{definition}

\begin{definition}[Reassignment of a Partition]
Fix $n,m \in \ZZ^+$. Let $\Fc = \{C_1,\ldots,C_k\}$ be a resolvable multiset with partitions. If $\Pc$ and $\Pc'$ are distinct partitions, not just formed by a trivial renaming of the partition classes, then call $\Pc'$ a reassignment of $\Pc$.
\end{definition}

\begin{example}
In Figure 3, $\Pc'$ is a reassignment of $\Pc$.
\end{example}

\begin{definition}[Uniquely Resolvable Multiset]
A resolvable multiset is uniquely resolvable if the partition $\Pc = \{A_1,\ldots,A_n\}$ corresponding to it is unique, up to a trivial renaming of the partition classes. For any $n,m \in \ZZ^+$, a uniquely resolvable multiset always exists. Namely, the multiset $\Fc$ containing $n$ copies of $[m]$ is uniquely resolvable.
\end{definition}

\begin{example}
Below are two examples of $\Fc$ that are uniquely resolvable and not uniquely resolvable respectively.
\begin{figure}[H]
    \centering
    \includegraphics[scale=0.3]{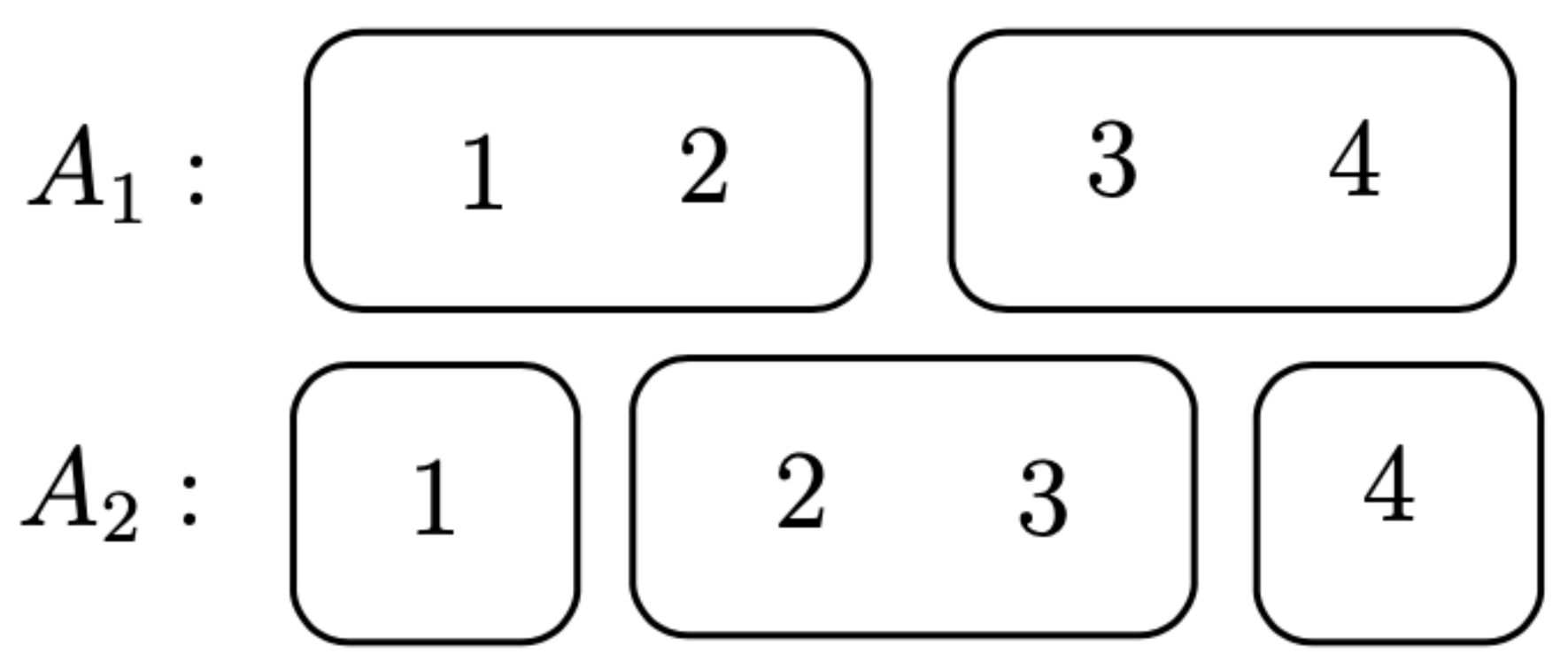}
    \caption{This corresponds to the case $(2,4)$. $\Fc$ contains the 5 sets $C_1 = \{1,2\}, C_2 = \{3,4\}, C_3 = \{1\}, C_4 = \{2,3\}$ and $C_5 = \{4\}$. It is clear that there is only one possible partition $\Pc = \{A_1,A_2\}$ here with $A_1 = \{1,2\}$ and $A_2 = \{3,4,5\}$. Therefore, this is a uniquely resolvable multiset. In fact, it also turns out to be extremal.}
    \label{fig2}
\end{figure}

\begin{figure}[H]
    \centering
    \includegraphics[scale=0.4]{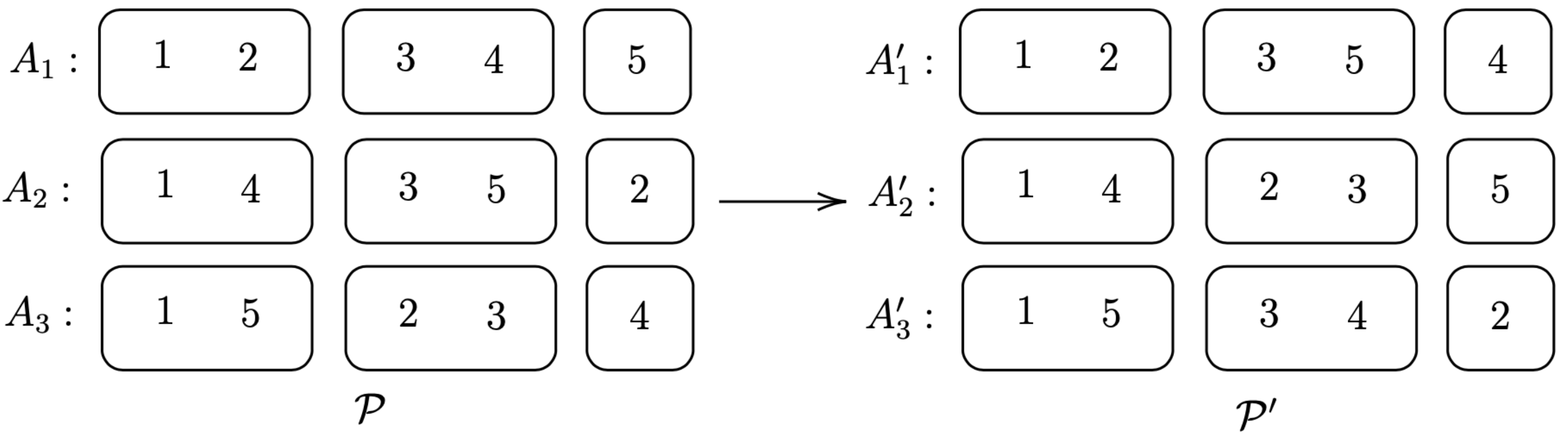}
    \caption{This corresponds to the case $(3,5)$. $\Fc$ contains the 9 sets $C_1 = \{1,2\}, C_2 = \{3,4\}, C_3 = \{5\}, C_4 = \{1,4\}, C_5 = \{3,5\}, C_6 = \{2\}, C_7 = \{1,5\}, C_8 = \{2,3\}$ and $C_9 = \{4\}$. There are two possible partitions. $\Pc = \{A_1,A_2,A_3\}$ with $A_1 = \{1,2,3\}, A_2 = \{4,5,6\}$ and $A_3 = \{7,8,9\}$ is a valid partition, but so is $\Pc' = \{A_1',A_2',A_3'\}$ with $A_1' = \{1,5,9\}, A_2' = \{4,8,3\}$ and $A_3' = \{7,2,6\}$. Thus, $\Fc$ is resolvable but not uniquely so.}
    \label{fig3}
\end{figure}
\end{example}

\begin{definition}[Extremal Uniquely Resolvable Multiset]
Fix $n,m \in \ZZ^+$. A uniquely resolvable multiset $\Fc$ is called extremal if for all uniquely resolvable multisets $\Fc'$, we have $|\Fc'| \leq |\Fc|$. For brevity, we will refer to these as extremal multisets.
\end{definition}

While a resolvable multiset is allowed to contain repetitions, the only set that can repeat without violating the existence of a unique partition is $[m]$ itself. We will later see that when $n \leq 2^{m-1}-1$, we can always remove $[m]$ from any extremal uniquely resolvable multiset and preserve its size (Lemma 17). So we can always find extremal uniquely resolvable multisets containing distinct non-empty proper subsets.

\begin{definition}[Classes of a Partition]
Let $\Pc = \{A_1,\ldots,A_n\}$ be a partition of $\Fc$. Then $\{C_j : j \in A_i\}$ is a class of $\Pc$ corresponding to $A_i$ for each $i$, and the $C_j$ are the components of $A_i$. We will abuse notation to say that $A_i$ itself is a class corresponding to $\Pc$. 
\end{definition}

\begin{definition}[Class Size]
Let $\Pc = \{A_1,\ldots,A_n\}$ be a partition of $\Fc$. Then $d_i = |A_i| \geq 1$ is the size of the class $A_i$ for each $i$.
\end{definition}
\begin{example}
In Figure 2, the class $A_1$ has size 2 (its 2 components are $\{1,2\}$ and $\{3,4\}$) and the class $A_2$ has size $3$ (its 3 components are $\{1\},\{2,3\}$ and $\{4\}$).
\end{example}

\begin{definition}[Proper subsets induced by a Class]
Let $\Fc = \{C_1,\ldots,C_k\}$ be a multiset and let $\Pc = \{A_1,\ldots,A_n\}$ be a partition of $\Fc$. Fix a class $A_i$. The proper subsets induced by $A_i$ are all the non-empty proper subsets of $[m]$ that can be formed by taking unions over some components in $A_i$. A size $d$ class induces $2^d-2$ proper subsets.
\end{definition}

\begin{example}
In Figure 2, $A_2$ induces the proper subsets $\{1\}$, $\{2,3\}$, $\{4\}$, $\{1,2,3\},\{1,4\}$ and $\{2,3,4\}$. 
\end{example}

\section{Determining $g(n,m)$}

\subsection{Introductory Results}

\begin{theorem} 
Suppose $n = 2$. Then for all $m \in \ZZ^+$,
\begin{equation}
g(2,m) = m+1
\end{equation}
\end{theorem}
\begin{proof}
Fix $n=2$. Let $\Fc = \{C_1,\ldots,C_m,C_{m+1}\}$ with $C_i = \{i\}$ for $i \in [m]$ and $C_{m+1} = [m]$. Clearly, there is only one valid partition $\Pc = \{[m],\{m+1\}\}$. Therefore, $g(2,m) \geq m+1$. Conversely, let $\Fc = \{C_1,\ldots,C_k\}$ be an extremal multiset, so $k = g(n,m)$. We will show that $k \leq m+1$.

Consider the graph $G$ whose vertices are $C_1,\ldots,C_k$, where $\{C_i,C_j\}$ for $i \neq j$ is an edge iff $C_i \cap C_j \neq \varnothing$. Since each number $i \in [m]$ appears exactly twice in the sets $C_1,\ldots,C_k$, $G$ has at most $m$ edges. We claim that $G$ must be connected, so $k \leq m+1$ as desired. 

Suppose by contradiction that $G$ has a proper connected component with vertex set $\{C_{i_1},\ldots,C_{i_l}\}$. Since $\Fc$ is resolvable into 2 classes, each element of $C_{i_1} \cup \ldots \cup C_{i_l}$ is in exactly two sets in $\{C_{i_1},\ldots,C_{i_l}\}$ and in no other set in $\{C_1,\ldots,C_k\} \setminus \{C_{i_1},\ldots,C_{i_l}\}$. Thus we can exchange the sets in $\{C_{i_1},\ldots,C_{i_l}\}$ that occur in one of the two partition classes with those in the other, contradicting the uniqueness of the partition.
\end{proof}

The result below provides a necessary condition for $\Pc$ to be the partition of a uniquely resolvable multiset $\Fc$, but it is not a sufficient condition.

\begin{lemma}[Subset Criterion]
Suppose $\Fc$ is a uniquely resolvable multiset with respect to $n,m \in \ZZ^+$. Let $\Pc$ be the unique partition of $\Fc$. Let $d_1,\ldots,d_n$ be the sizes of the classes in $\Pc$. Then we have that
\begin{equation}
\sum_{i=1}^n (2^{d_i} - 2) \leq 2^m - 2
\end{equation}
\end{lemma}

\begin{proof}
We will count the number of proper subsets of $[m]$ formed by each class. For any class $i$, there are $d_i \geq 1$ components. Therefore, by combining components, we can produce $2^{d_i} - 2$ proper subsets of $[m]$. Suppose the same proper subset of $[m]$ can be formed using only components of class $i$ and also using only components of class $j$ (with $i \neq j$). Then we can move the components forming that proper subset of $[m]$ in class $i$ to class $j$ and vice versa. This contradicts the fact that $\Pc$ is a unique partition of $\Fc$. So each class must contribute $2^{d_i} - 2$ new proper subsets. The total number of such subsets is given by $\sum_{i=1}^n (2^{d_i} - 2)$. This must be less than or equal to $2^m - 2$, the number of proper subsets of $[m]$, because otherwise some two classes must have at least one common proper subset by the pigeonhole principle.
\end{proof}

\begin{remark}[Choice of Regime]
Suppose $n > 2^{m-1} - 1$. Observe that $d_i \geq 2$ for at most $2^{m-1} - 1$ values of $i \in [n]$. This means $d_i = 1$ for at least $n - 2^{m-1} + 1$ values of $i$, so any uniquely resolvable multiset corresponding to $(n, m)$ must be formed from one corresponding to $(2^{m-1}-1,m)$ by adding $n - 2^{m-1} + 1$ copies of the entire set $[m]$. Therefore, when $n > 2^{m-1}-1$, we trivially have $g(n,m) = g(2^{m-1},m) + n - 2^{m-1} + 1$. We thus focus on the regime $n \leq 2^{m-1} - 1$, which has a lot more interesting structure.
\end{remark}

\subsection{Lower Bounds for $g(n,m)$ when $n \leq 2^{m-1} - 1$}
The focus of this section is to show a lower bound for $g(n,m)$. In the next section, we will provide close upper bounds in a large regime.

\begin{lemma}
Suppose $n \leq 2^{m-1}-1$. Then we have that
\begin{equation}
g(n,m) \geq 2n
\end{equation}
\end{lemma}

\begin{proof}
We will construct a uniquely resolvable multiset $\Fc$ with each class in the partition $\Pc$ having size $2$. Specify each class $A_i$ in $\Pc$ (for each $i \in [n]$) by a pair of disjoint proper subsets $\{C_1^{(i)},C_2^{(i)}\}$ of $[m]$ such that $C_1^{(i)} \cup C_2^{(i)} = [m]$. There are $\frac{2^m - 2}{2} = 2^{m-1}-1$ such pairs, so we can specify up to $2^{m-1} - 1$ classes, each of size 2, without repeating any proper subsets. The resultant multiset is $\Fc = \{C_{j}^{(i)} : i \in [n], j \in \{1,2\}\}$. We claim that $\Pc$ is the only possible partition.

If there was any reassignment of the components to form a different collection of classes, we must still have exactly two components in each class. If not, then some class must have at least 3 components, so some other class must have exactly one component, which must be $[m]$. But no class had a single component of size $m$ before the reassignment because we were working only with proper subsets. For any component $C$ in class of size 2, it must be paired with its complement, which was precisely the original partition $\Pc$. So we have that $g(n,m) \geq 2n$.
\end{proof}

\begin{remark}[Regime of $n \geq 2^{m-1}-1$]
Lemma 15 above shows that $g(2^{m-1}-1,m) \geq 2^m - 2$. Let $\Fc$ be extremal multiset corresponding to $(2^{m-1}-1,m)$. Lemma 17 below implies that we can assume every class in $\Fc$ has size at least 2. If even a single class had size $3$ or more, then the subset criterion would be violated. So $g(2^{m-1}-1,m) \leq 2^m - 2$. By Remark 14, when $n \geq 2^{m-1}-1$, $g(n) = g(2^{m-1}-1,m) + n - 2^{m-1} + 1 = n + 2^{m-1} - 1$.
\end{remark}

\begin{lemma}
Suppose $n \leq 2^{m-1}-1$. Then there exists an extremal multiset $\Fc$ with a unique partition $\Pc$ such that each of the $n$ classes in $\Pc$ has size at least 2.
\end{lemma}

\begin{proof}
Let $\Fc_1$ be an extremal multiset with a unique partition $\Pc_1$ such that at least one class has size 1. Then we must have at least one class with size at least 3. If not, every class has size at most $2$ so the number of sets in $\Fc$ is less than the number of sets in the case when every class has size 2, a contradiction to Lemma 15.

Let $A_1$ and $A_2$ be the classes with size 1 and size $l \geq 3$ respectively. So $A_1$ has only a single component $[m]$. Let the $l$ components of $A_2$ be $C_1^{(b)},C_2^{(b)},C_3^{(b)},\ldots,C_l^{(b)}$. Then reduce the number of sets in $A_2$ by combining $C_1{(b)}$ and $C_2{(b)}$. Then we can specify $A_1$ with 2 sets by breaking $[m]$ up into the corresponding components $C_1^{(a)}$ and $\cup_{i=2}^l C_i^{(a)}$.

We claim that this new multiset $\Fc_2$ with its partition $\Pc_2$ is still uniquely resolvable. Suppose by contradiction that there exists another partition $\Pc_2'$ for $\Fc_2$. We will obtain a new partition $\Pc_1'$ for $\Fc_1$, contradicting the fact that $\Pc_1$ was unique.

Any such $\Pc_2'$ would have to involve dividing the sets of $A_1$ into different classes, because otherwise we directly have a reassignment of $\Pc_1$ as well. So we have that $\cup_{i=2}^l C_i^{(a)}$ is paired with components from some other classes: let us call the union of these components $K$.
\begin{itemize}
\item[1.] Move $C_1^{(a)}$ to where $K$ is.
\item[2.] Combine the sets $\cup_{i=2}^k C_i^{(a)}$ and $C_1^{(a)}$ and split the sets $C_1^{(b)}$ and $C_2^{(b)}$.
\item[3.] Move $K$ to wherever $C_1^{(b)}$ is.
\item[4.] Move $C_1^{(b)}$ to wherever $C_1^{(a)}$ was.
\end{itemize}

\begin{figure}[H]
\centering
\includegraphics[scale=0.4]{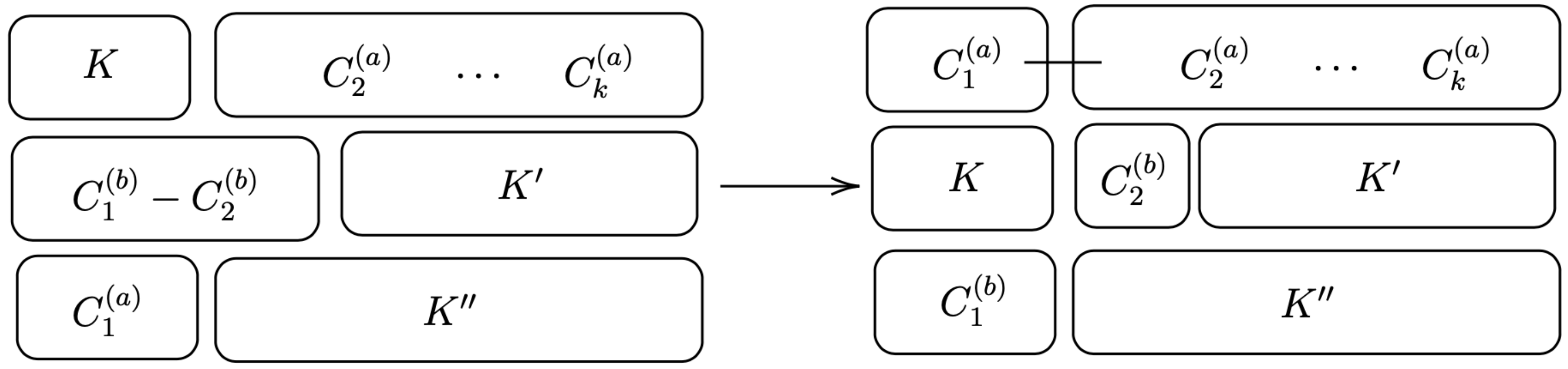}
\caption{Visual representation of the new assignment.}\label{fig4}
\end{figure}
Since $A_1$ is now unaffected and $C_1^{(b)}$, $C_2^{(b)}$ are separated, this induces a reassignment of $\Pc_1$, a contradiction. It follows that the newly formed multiset $\Fc_2$ must also be uniquely resolvable and extremal as we have not changed the number of subsets. We have now eliminated the class of size 1. We can repeat this process as many times as required. This completes the proof.
\end{proof}


\begin{remark}
Lemma 17 shows us that when $n \leq 2^{m-1}-1$, even if we include $[m]$ in an extremal multiset $\Fc$, we can eliminate it.
\end{remark}

\begin{definition}
Fix $2 \leq k \leq m$. Let $P_k(m)$ be the largest possible integer $N$ such that there exists a uniquely resolvable multiset $\Fc$ corresponding to $(N,m)$ with a partition $\Pc$ that contains $N$ classes of size $k$ each.
\end{definition}
It is easy to see that $P_k(m) \geq P_{k+1}(m)$. We just observe that for any uniquely resolvable multiset $\Fc$ corresponding to $(P_{k+1}(m),m)$ with $P_{k+1}(m)$ classes of size $k+1$, combing any two components of every class produces a uniquely resolvable multiset $\Fc'$ corresponding to $(P_{k+1}(m),m)$ with every class of size $k$.

We note that \cite{colbourn2016disjoint} studies a related problem on finding a multiset of partitions of $[m]$ into $k$ components such that no two components from different partitions are the same. While similar in spirit, considering multisets that are uniquely resolvable makes our setting quite different. 

\begin{lemma}
For any set $S$ with $N$ elements, there exist a set of $2^{N-1} - 1$ proper subsets of $S$ with each set containing at least $\lceil\frac{N}{2}\rceil$ elements such that no two sets are complements of each other.
\end{lemma}
\begin{proof}
If $N$ is odd, this follows by simply counting all possible subsets of size $\lceil\frac{N}{2}\rceil$ or greater because the sum of the cardinalities of any two such sets is more than $N$. If $N$ is even, we can construct a family with all proper subsets of $S$ with size $\frac{N}{2}+1$ or larger. For the sets of size $\frac{N}{2}$, we can pair them up with their complements and add exactly one set from each pair to our family. This must also have size $2^{N-1}-1$.
\end{proof}

\begin{lemma}
Choose an integer $k$ such that $2 \leq k < m$ and $\frac{m}{k-1} \geq 2$. Then
\begin{equation}
P_k(m) \geq 2^{\lfloor\frac{m}{k-1}\rfloor-1} - 1
\end{equation}
\end{lemma}

\begin{proof}
For simplicity, assume that $k-1$ divides $m$. The proof for the other cases is similar.

\begin{enumerate}
\item Divide $[m]$ into $k-1$ sets: $S_1 = \{1,\ldots,\frac{m}{k-1}\}, S_2 = \{\frac{m}{k-1} + 1,\ldots,\frac{2m}{k-1}\}, \ldots, S_{k-1} = \{(k-2)\frac{m}{k-1} + 1,\ldots,m\}$. 
\item Take $l = 2^{\frac{m}{k-1}-1} - 1$ proper subsets of each $S_i$ with size at least $\lceil \frac{m}{2(k-1)} \rceil \geq 1$ such that no two are complements of each other by by Lemma 20. For each $i \in [k-1]$, label the sets $X^{(i)}_1,\ldots,X^{(i)}_l$.
\item Construct $l$ tuples of sets of the form $(X^{(1)}_j,\ldots,X^{(k-1)}_j)$ for each $j \in [l]$. 
\item We now construct $l$ classes each of size $k$. Choose any tuple $(X^{(1)},\ldots,X^{(k-1)})$. Let $C = \cup_{i=1}^{k-1} X^{(i)}$. So $C$ contains at least one element in $S_i$ for each $1 \leq i \leq k-1$. For each $i$, let $C_i = S_i \setminus C$ be the set of remaining elements in $S_i$. Then construct a class with the sets $\{C,C_1,\ldots,C_{k-1}\}$. Call $C$ the \textit{central component} of this class. Repeat this process for each tuple to create $l = 2^{\frac{m}{k-1}-1} - 1$ classes of size $k$.
\end{enumerate}

This construction clearly induces a partition $\Pc$ for a multiset $\Fc$ corresponding to $(2^{\frac{m}{k-1}-1}-1,m)$ with every class of size $k$. We claim $\Fc$ is uniquely resolvable. 


Consider any central component $C$. Assume that originally in $\Pc$, $C$ was in the class $\{C,C_1,\ldots,C_{k-1}\}$. Since $C$ does not contain all the elements in $S_1$ by construction, it requires at least one set $C_1' \subseteq S_1$ in the same class even in $\Pc'$. It cannot have more than one subset that is a subset of $S_1$, because by the Pigeonhole principle, this would mean some other central component has no subset of $S_1$ in its class in $\Pc'$. So in fact $C$ has exactly one set $C_1' \subseteq S_1$ in its class in $\Pc'$. In fact, $C_1' = C_1$ because $C_1'$ must contain all the elements of $S_1$ not in $C$ without having an intersection with $C$, so we must have $C_1' = S_1 \setminus C = C_1$. Repeating the same argument for $i=2,\ldots,k-1$, $C$ must be in the same class in both $\Pc$ and $\Pc'$. Repeating the argument for every central component, we have that $\Pc = \Pc'$. It follows that $\Fc$ is uniquely resolvable.

Consider any central component $C$. Assume that originally in $\Pc$, $C$ was in the class $\{C,C_1,\ldots,C_{k-1}\}$. Since $C$ does not contain all the elements in $S_1$ by construction, it requires at least one set $C_1' \subseteq S_1$ in the same class even in $\Pc'$. It cannot have more than one subset that is a subset of $S_1$, because by the Pigeonhole principle, this would mean some other central component has no subset of $S_1$ in its class in $\Pc'$. So in fact $C$ has exactly one set $C_1' \subseteq S_1$ in its class in $\Pc'$. In fact, $C_1' = C_1$ because $C_1'$ must contain all the elements of $S_1$ not in $C$ without having an intersection with $C$, so we must have $C_1' = S_1 \setminus C = C_1$. Repeating the same argument for $i=2,\ldots,k-1$, $C$ must be in the same class in both $\Pc$ and $\Pc'$. Repeating the argument for every central component, we have that $\Pc = \Pc'$. It follows that $\Fc$ is uniquely resolvable.

When $k-1$ does not divide $m$, a similar construction works. If $r$ is the remainder when $m$ is divided by $k-1$, just add the same $r$ integers into every central component.
\end{proof}

\begin{lemma}
Suppose $2 \leq k < m$ and $\frac{m}{k-1} \geq 2$. Then for any $n \leq 2^{\lfloor\frac{m}{k-1}\rfloor-1} - 1$,
\begin{equation}
    g(n,m) \geq kn
\end{equation}
\end{lemma}
\begin{proof}
By Lemma 21, if $n \leq 2^{\lfloor\frac{m}{k-1}\rfloor-1} - 1$, we can construct a uniquely resolvable multiset $\Fc$ corresponding to $(n,m)$ with every class having size $k$, so $|\Fc| = kn \leq g(n,m)$.
\end{proof}

\begin{theorem}[Lower Bound]
Suppose $1 \leq n \leq 2^{m-1}-1$. Then 
\begin{equation}
g(n,m) > \frac{n(m+1)}{\log_2(n+1)+2} = \Omega\left(\frac{nm}{\log_2 n}\right)
\end{equation}
\end{theorem}
\begin{proof}
From Lemma 22, it suffices to choose $k$ such that $n \leq 2^{\frac{m-(k-2)}{k-1}-1} - 1 = 2^{\frac{m+1}{k-1}-2} - 1$. $k = \left\lfloor \frac{m+1}{2 + \log_2(n+1)} \right\rfloor + 1$ satisfies this inequality, so 
\[g(n,m) \geq n \left(\left\lfloor \frac{m+1}{\log_2(n+1)+2} \right\rfloor + 1\right) > \frac{n(m+1)}{\log_2(n+1)+2}\]
\end{proof}

\subsection{Upper Bounds for $g(n,m)$ when $n \leq 2^{m-1} - 1$}

The natural question to ask is how close to optimal the lower bound from Section 3.2 is. We provide an upper bound in a large regime, namely, when $n = 2^{cm}$ for some $c \in (0,1)$ which matches the lower bound in Theorem 23 up to a $\log_2(\frac{1}{c})$ factor.

\begin{lemma}
\cite{flum} Let $H_2(x)$ be the binary cross entropy function.
\begin{equation}
H_2(x) = -x\log_2(x) - (1-x)\log_2(1-x)
\end{equation}
Suppose $k \geq 2$. Then
\begin{equation}
\sum_{i=0}^{\lfloor \frac{m}{k} \rfloor} {m\choose{i}} \leq 2^{m H_2(\frac{1}{k})}
\end{equation}
\end{lemma}

\begin{lemma}
For any $m$ and $3 \leq k \leq \frac{m}{2}$,
\begin{equation}
P_k(m) \leq 2^{mH_2(\frac{1}{k})}
\end{equation}
\end{lemma}
\begin{proof}
Consider a uniquely resolvable multiset $\Fc$ corresponding to $(P_k(m),m)$ with partition $\Pc$ where every class has size $k$. Note that each class must have at least one subset with at most $\lfloor \frac{m}{k} \rfloor$ elements. If $\Fc$ is uniquely resolvable, the only subset it can contain multiple times is $[m]$, so each subset with size at most $\lfloor \frac{m}{k} \rfloor$ appears at most once in $\Fc$. Therefore, by Lemma 24,
$ |\Fc| \leq \sum_{i=0}^{\lfloor \frac{m}{k} \rfloor} {m \choose{i}}
\leq 2^{mH_2(\frac{1}{k})}$
\end{proof}
\begin{lemma}
Let $k \geq 2$ be an integer such that $n > 2^{m H_2(\frac{1}{k})}$. Then $g(n,m) \leq kn$.
\end{lemma}
\begin{proof}
Suppose by contradiction that $g(n,m) > kn$. Consider an extremal multiset $\Fc$ corresponding to $(n,m)$ with partition $\Pc$ and class sizes $d_1,\ldots,d_n$. For all $i \in [n]$, let $r_i = \max(0,d_i - (k-1))$.
Note that $\sum_{i=1}^n r_i \geq n$ because otherwise $g(n,m) \leq kn$. Consider any class with positive $r_i$. So we know $d_i = k + r_i - 1$. We claim that this class must contain at least $r_i$ subsets with size at most $\lfloor \frac{m}{k} \rfloor$.
If not, this class must contain at most $r_i-1$ subsets with size at most $\lfloor \frac{m}{k} \rfloor$, so at least $k$ subsets must have size at least $\lfloor \frac{m}{k} \rfloor + 1$. It follows that the sum of the cardinalities of the sets in the class is at least
\[k\left(\left\lfloor \frac{m}{k} \right\rfloor + 1\right) \geq k\frac{m-k+1}{k} + k = m+1\]
This is a contradiction. Therefore, each class must contain at least $r_i$ subsets with size at most $\lfloor \frac{m}{k} \rfloor$. Since a uniquely resolvable multiset must contain at most $2^{m H_2(\frac{1}{k})}$ subsets with size at most $\lfloor \frac{m}{k} \rfloor$ by Lemma 25, it follows that $n \leq \sum_{i=1}^n r_i \leq 2^{m H_2(\frac{1}{k})}$, a contradiction.
\end{proof}

\begin{theorem}[Upper Bound]
Let $c \in (0,1)$ and $n = 2^{cm}$ with $1 \leq n \leq 2^{m-1} - 1$. Then
\begin{equation}
g(n,m) \leq \frac{n}{c} \left(6 - 3.2\log_2(c)\right)
\end{equation}
\end{theorem}
\begin{proof}
We will make use of Lemma 26 and find $k$ such that $n = 2^{cm} > 2^{m H_2(\frac{1}{k})}$. Note that if $k \geq 3$,
\[H_2\left(\frac{1}{k}\right) < \frac{2\log_2(k)}{k}\]
So it suffices to find an integer $k \geq 3$ such that $\frac{2\log_2(k)}{k} \leq c$, or equivalently $\frac{\ln(k)}{k} \leq \frac{c\ln2}{2}$. Let $\epsilon = \frac{c\ln2}{2} \in (0,\frac{\ln2}{2})$. We claim that $k = \lceil \frac{1.6}{\epsilon} \ln\left(\frac{1}{\epsilon}\right) \rceil$ suffices. We note that $k \geq 3$ because
$\frac{1.6}{\epsilon} \ln\left(\frac{1}{\epsilon}\right) \geq  \frac{3.2}{\ln2} \ln\left(\frac{2}{\ln2}\right) \geq 4$. We want to show that $\frac{\ln k}{k} \leq \epsilon$. Since $\frac{\ln x}{x}$ is decreasing for $x \geq 3$, it suffices to show that
\begin{align*}
\frac{\ln\left(\frac{1.6}{\epsilon} \ln\left(\frac{1}{\epsilon}\right)\right)}{\frac{1.6}{\epsilon} \ln\left(\frac{1}{\epsilon}\right)} \leq \epsilon
&\iff 0.6\ln\left(\frac{1}{\epsilon}\right) - \ln\left(\ln\left(\frac{1}{\epsilon}\right)\right) - \ln(1.6) \geq 0
\end{align*}
Define $h(x) = 0.6\ln(x) - \ln(\ln(x)) - \ln(1.6)$. It suffices to show $h(x)$ is positive on $(0,\infty)$. The only critical point of $h(x)$ is at $x = e^{\frac{5}{3}}$. It is easy to verify that $h(e^{\frac{5}{3}}) > 0$ and $h''(e^{\frac{5}{3}}) > 0$. So $h(x)$ is always positive. It follows that
\begin{align*}
g(n,m)
&\leq n \left\lceil \frac{1.6}{\epsilon} \ln\left(\frac{1}{\epsilon}\right) \right\rceil\\
&= n \left\lceil \frac{3.2}{c\ln2} \ln\left(\frac{2}{c\ln2}\right) \right\rceil\\
&\leq \frac{n}{c} \left(6 - 3.2\log_2(c)\right)
\end{align*}
\end{proof}

\begin{remark}
This upper bound of $\frac{n}{c} (6-3.2\log_2(c)) = O(\frac{n}{c}\log_2(\frac{1}{c}))$ is in the regime $n = 2^{cm}$ for any $c \in (0,1)$. If we look at our lower bound from Theorem 23 in this regime, since $\log_2(n) = cm$, we have that $g(n,m) \geq \Omega(\frac{nm}{\log_2(n)}) = \Omega(\frac{n}{c})$. Therefore, our bounds are tight up to a $\log_2(\frac{1}{c})$ factor.
\end{remark}

\subsection{Determining $g(n,m)$ when $2^{m-1} - O(2^{\frac{m}{2}}) \leq n \leq 2^{m-1}-1$}

Building on ideas from the previous sections, we can in fact obtain some exact results for $g(n,m)$ when $n$ is close to $2^{m-1}-1$.

\begin{theorem}
Suppose $m \geq 4$, and $2^{m-1}+1-2^{\lfloor\frac{m}{2}\rfloor} \leq n \leq 2^{m-1} - 1$. Then
\begin{equation}
g(n,m) = 2n + \left\lfloor \frac{2^{m-1}-1-n}{2} \right\rfloor
\end{equation}
\end{theorem}
\begin{proof}

Pick a maximum multiset $\Fc$ with partition $\Pc$. By Lemma 17, without loss in generality, each of the $n$ classes in $\Pc$ has size at least 2. Suppose $k'$ classes have at least size 3. Let $k = \lfloor\frac{2^{m-1}-1-n}{2}\rfloor$. By the subset criterion, $k'(2^3-2) + (n-k')(2^2-2) \leq 2^m - 2$, so $k' \leq k$. We will argue that in fact choosing $k' = k$ size 3 classes and $n-k$ size 2 classes produces an extremal multiset. Note that by Lemma 21, $P_3(m) \geq 2^{\lfloor\frac{m}{2}\rfloor-1}-1$, so our condition on $n$ ensures that $k \leq P_3(m)$. 

$\Longrightarrow$ We first show that $g(n,m) \geq 2n+k$. Pick $k$ size 3 classes that are uniquely resolvable. This can be done because $k \leq P_3(m)$. Each size 3 class induces 6 proper subsets that form complementary pairs. Thus the remaining proper subsets form $\frac{2^m-2-6k}{2} \geq n - k$ complementary pairs and so the remaining $n-k$ classes can be specified with size 2. We claim this construction $(\Fc,\Pc)$ is uniquely resolvable. Suppose not. Then there exists some reassignment $\Pc'$. 

Suppose $\Pc'$ contains a class of size 4 or more. Since the total number of subsets in the multiset must be preserved, we must have a component $C \in \Fc$ that was originally in a size 3 class $\{C,C_1,C_2\}$ in $\Pc$ which is now in a size 2 class in $\Pc'$ as $\{C,C^c\}$. (We cannot have size 1 classes because $[m] \not\in \Fc$.) This means $C^c \in \Fc$. $C^c = C_1 \cup C_2$ could not have appeared as a single component in any of the original $k$ size 3 classes as we could swap it with $C_1 \cup C_2$. Further, we only introduced $n-k$ subset pairs which did not contain the proper subsets of the first $k$ classes. So $C^c \not\in \Fc$, a contradiction. Therefore, $\Pc'$ must also only have classes of size $2$ or $3$.

Pick a component $C$ that was earlier in a size 3 class. If it is now in a size 2 class, then $C^c$ must exist as a single component. By the same argument as above, $C^c$ cannot exist. So $C$ must be in a size 3 class again. So $\Pc'$ has at least $k$ size 3 classes, and cannot have more than $k$ such classes because the number of subsets must be preserved. In fact, the assignment of the components to the $k$ size 3 classes must be the same because the original $k$ classes was a unique partition. The remaining $n-k$ classes must all be of size 2. No reassignment of just size 2 classes is possible because every class must be of the form $\{C,C^c\}$. Therefore, $\Pc$ is unique and $\Fc$ is uniquely resolvable. So $g(n,m) \geq |\Fc| = 3k + 2(n-k) = 2n+k$.

$\Longleftarrow$ We show that $g(n,m) \leq 2n + k$. Consider a maximum multiset $\Fc$ with partition $\Pc$ corresponding to $(n,m)$. By Lemma 17, we can assume that each of the $n$ classes has size $d_i \geq 2$ with at most $k$ classes of size 3 or more, so at least $n-k$ classes have size exactly 2. Without loss in generality, assume the last $n-k$ classes have size $2$. By the subset criterion, $\sum_{i=1}^n (2^{d_i}-2) = 2(n-k) + \sum_{i=1}^k (2^{d_i} - 2) \leq 2^m - 2$. Let us maximize
\begin{equation}
\sum_{i=1}^{k} d_i
\end{equation}
subject to the constraints that
\begin{equation}
d_i \geq 2, d_i \in \ZZ \text{ for all } i \in [k]
\end{equation}
\begin{equation}
\sum_{i=1}^k 2^{d_i} \leq 2^m - 2 - 2n + 4k
\end{equation}

Setting $d_i = 3$ for $i \in [k]$ is a feasible solution, which results in $\sum_{i=1}^k d_i = 3k$. Consider any optimal solution. Suppose $d_{i'} = 2$ for some $i' \in [k]$. If $d_i \leq 3$ for all $i \in [k]$, then our objective is less than $3k$, a contradiction. Since $f(d) = 2^d$ is convex and strictly increasing, $f(d'+1) + f(d''-1) < f(d') + f(d'')$ for any $d' \leq d''-2$. If any $d_{i''} > 3$, we can replace $d_{i'},d_{i''}$ with $d_{i'}+1,d_{i''}-1$. So we may assume $d_i \geq 3$ for $i \in [k]$. If in fact $d_{i'} > 3$ for some $i'$, by our choice of $k$, the constraint in Equation 14 is violated. So $g(n,m) \leq \sum_{i=1}^n d_i \leq 3k + 2(n-k) = 2n + k$.

\end{proof}

\begin{remark}
By pairing sets more carefully for the central components in Lemma 21's construction when $k=3$, we can in fact show $P_3(m) \geq 2^{\lfloor \frac{m}{2} \rfloor} - 2$ for $m \geq 6$. Applying this to Theorem 29, we can slightly improve this result and show that when $m \geq 6$, $g(n,m) = 2n+\lfloor\frac{2^{m-1}-1-n}{2}\rfloor$ whenever $2^{m-1}+3-2^{\lfloor \frac{m}{2} \rfloor+1} \leq n \leq 2^{m-1}-1$. We omit this modification for brevity.
\end{remark}

\section{Discussion}
In the regime $n \leq 2^{m-1}$, we provide a lower bound for $g(n,m)$ of the form $\Omega(\frac{nm}{\log_2 n})$. When $n = 2^{cm}$ for any $c \in (0,1)$, this lower bound simplifies to $\Omega(\frac{n}{c})$, and we provide a matching upper bound up to a $\log_2(\frac{1}{c})$ factor, thus showing that our bounds are near optimal for most values of $n$. We also compute some exact values when $n \geq 2^{m-1} - O(2^{\frac{m}{2}})$. It would be interesting to see sharp bounds for $g(n,m)$ when $n$ is a polynomial in $m$. Additionally, we focused on finding the largest possible size of a non-uniform 1-balanced block design on $[m]$ with $\lambda = n$ that is uniquely resolvable. A natural generalization is to find bounds on non-uniform $t$-balanced block designs. 



\section*{Acknowledgements}
The author would like to thank Hunter Novak Spink and Ryan Alweiss for their insightful comments and corrections, as well as Nikil Roashan Selvam, Bruce Rothschild, Adam Moreno and Rajeshwari Jadhav for helpful discussions and encouragement. The author is also grateful to the anonymous reviewer for their detailed comments and suggestions, especially regarding Theorems 12 and 29.

\bibliographystyle{unsrt}
\bibliographystyle{abbrv}










\end{document}